\documentclass[11pt]{article}  

\usepackage{amssymb,amsmath,amsthm,amscd}
\usepackage[all]{xy}
\usepackage[T1]{fontenc}
\usepackage{lmodern}
\usepackage[latin1]{inputenc}
\usepackage{color}
\usepackage[latin1]{inputenc}
\usepackage{array}
\usepackage{stmaryrd}
\usepackage[pdftex]{hyperref} 

\usepackage{enumerate}

\newtheoremstyle{fancy1}{10pt}{10pt}{\itshape}{12pt}{\textsc\bgroup}{.\egroup}{8pt}{}

\newtheoremstyle{fancy2}{10pt}{10pt}{}{12pt}{\itshape}{.}{8pt}{ }

\newtheoremstyle{fancy3}{10pt}{10pt}{}{12pt}{\itshape}{ }{8pt}{ }

\theoremstyle{fancy1}
\newtheorem{thm}{Theorem}

\newtheorem{cor}[thm]{Corollary}
\newtheorem{lem}[thm]{Lemma}
\newtheorem{Isotropy Lemma}{Isotropy Lemma}
\newtheorem*{Connectedness Lemma*}{Connectedness Lemma}
\newtheorem{prop}[thm]{Proposition}

\newtheorem{main}{Theorem}
\newtheorem*{main*}{Theorem}

\newtheorem*{cor*}{Corollary}

\newtheorem*{claim*}{Claim}

\newtheorem*{problem*}{Problem}

\theoremstyle{fancy2}

\newtheorem{rem}[thm]{Remark}
\newtheorem*{rem*}{Remark}
\newtheorem{example}{Example}

\newcommand{\cref}[1]{Corollary~\ref{#1}}

\newcommand{\eref}[1]{Example~\ref{#1}}
\newcommand{\lref}[1]{Lemma~\ref{#1}}
\newcommand{\pref}[1]{Proposition~\ref{#1}}
\newcommand{\rref}[1]{Remark~\ref{#1}}
\newcommand{\tref}[1]{Theorem~\ref{#1}}
\newcommand{\sref}[1]{Section~\ref{#1}}

\newcommand{\ga}{\alpha}
\newcommand{\gb}{\beta}

\newcommand{\gt}{\theta}

\newcommand{\gC}{\Gamma}
\newcommand{\gS}{\Sigma}

\newcommand{\RP}{\mathbb{R\mkern1mu P}}
\newcommand{\CP}{\mathbb{C\mkern1mu P}}

\newcommand{\Sph}{\mathbb{S}}

\newcommand{\C}{{\mathbb{C}}}
\newcommand{\R}{{\mathbb{R}}}
\newcommand{\Z}{{\mathbb{Z}}}

\newcommand{\QH}{{\mathbb{H}}}

\newcommand{\I}{\ensuremath{\operatorname{I}}}

\newcommand{\G}{\ensuremath{\operatorname{G}}}
\newcommand{\D}{\ensuremath{\operatorname{D}}}
\newcommand{\SO}{\ensuremath{\operatorname{SO}}}

\renewcommand{\O}{\ensuremath{\operatorname{O}}}

\newcommand{\U}{\ensuremath{\operatorname{U}}}
\newcommand{\SU}{\ensuremath{\operatorname{SU}}}

\newcommand{\Pin}{\ensuremath{\operatorname{Pin}}}
\newcommand{\T}{\ensuremath{\operatorname{T}}}

\newcommand{\lb}{\llbracket}
\newcommand{\rb}{\rrbracket}

\newcommand{\B}{{B}}

\newcommand{\fp}{{\mathcal{P}}}
\newcommand{\fq}{{\mathcal{Q}}}

\newcommand{\fN}{{\mathcal{N}}}
\newcommand{\fB}{{\mathcal{B}}}
\newcommand{\fW}{{\mathcal{W}}}

\def\con#1=#2(#3){#1 \equiv #2 \bmod{#3}}

\newcommand{\hra}{\circlearrowright}

\newcommand{\ra}{\rightarrow}

\newcommand{\diag}{\ensuremath{\operatorname{diag}}}

\DeclareMathOperator{\Id}{Id}

\DeclareMathOperator{\Dic}{Dic}

\newcommand{\wt}{\widetilde}

\newcommand{\ovl}{\overline}

\newcommand{\gk}{\kappa}

\newcommand{\pig}{principal isotropy group }

\makeindex

\usepackage{psfrag}

\title{Nonnegatively curved 5-manifolds with non-abelian symmetry}
\author{Fabio Simas\footnote{The author was partially supported by CNPq-BRAZIL.}}

\date{}

\begin{document}
\maketitle



Known examples of manifolds which admit metrics of positive (sectional) curvature are rare when compared with nonnegatively curved examples. In fact, besides  rank one symmetric spaces, compact manifolds with positive curvature are known to exist only in dimensions below~25, while to generate new nonnegatively curved manifolds from known ones it is enough, for example, to take products, quotients or biquotients (see \cite{wolfgang2} for a survey). 

By the Soul Theorem any complete non-compact nonnegatively curved manifold is diffeomorphic to a vector bundle  over a compact manifold with nonnegative curvature. For compact manifolds of positive curvature Bonnet-Myers implies that the fundamental group is finite and in nonnegative curvature a finite cover is diffeomorphic to the product of a torus with a compact simply-connected manifold of nonnegative curvature~(see~\cite{ebin}). We will hence only consider compact simply-connected manifolds.

Recently, positively and nonnegatively curved manifolds were studied under the additional assumption of having a ``large'' isometry group (see the surveys \cite{grove-survey} and \cite{wilking}). The beginning of this subject was the result by Hsiang and Kleiner \cite{Kleiner1989} that a compact simply-connected 4-dimensional Riemannian manifold with positive curvature and $\Sph^1$-symmetry must be either $\Sph^4$ or $\CP^2$. The possible isometric circle-actions were classified in \cite{kerin} and \cite{grove-wilking}.

The classification of isometric circle actions on positively curved 5-manifolds is a very difficult problem and at the moment seems out of reach. In 2002, Rong~\cite{Rong2002} showed that a positively curved compact simply-connected 5-dimensional manifold with a 2-torus acting by isometries is diffeomorphic to a 5-sphere. 
In 2009 Galaz-Garcia and Searle~\cite{Galaz2011}, only assuming nonnegative curvature, showed that a simply-connected 5-manifold which admits an isometric action of a 2-torus is diffeomorphic to either $\Sph^5$, $\Sph^3 \times \Sph^2$, the nontrivial $\Sph^3$-bundle over $\Sph^2$, denoted by $\Sph^3 \tilde{\times} \Sph^2$, or the Wu-manifold $\fW=\SU(3)/\SO(3)$. The description of the actions is not yet solved in any of these cases.

In this context, a question that naturally arises is which 5-manifolds admit a metric of nonnegative (or positive) curvature with symmetry containing a connected non-abelian group $G$. 
In this paper we will classify such manifolds with nonnegative curvature and obtain a partial classification in positive curvature. For this purpose, we first classify all five-dimensional compact simply-connected manifolds which admit an action of a connected non-abelian Lie group without any geometric assumptions. They are either $\Sph^5$, $\Sph^3 \times \Sph^2$,  $\Sph^3\, \wt{\times}\, \Sph^2$, connected sums of $\Sph^3 \times \Sph^2$, or connected sums~$k\, \fW \,\#\, l\, \fB$ of copies of the Wu-manifold~$\fW$ and the Brieskorn variety $\fB$ of type $(2,3,3,3)$. Since any non-abelian connected Lie group contains $\SO(3)$ or $\SU(2)$ as a subgroup, it is natural to classify in addition the actions by these two groups up to equivariant diffeomorphisms. This is the content of Theorems \ref{thm exceptional SU(2)} and \ref{thm main intro} below.

To describe the actions we introduce the following key construction.
\vspace{0.3cm}

\noindent{\it Main example.} Let $m \leq n$ and $l$ be nonnegative integers and consider the $\Sph^1$-action on $\SU(2) \times \Sph^3 = \Sph^3 \times \Sph^3$ given by 
$$x\cdot (p, (z,w)) = (px^l, (x^m z,x^n w)),$$
where we regard $\SU(2)$ as the group of unit quaternions, $p \in \SU(2)$,   $x \in \Sph^1 =\{e^{i\gt} \in  \SU(2)\}$ and $(z,w) \in \Sph^3 \subset \C^2$.
This action is free whenever $\gcd(l,m) = \gcd(l,n) = 1$. Notice that $l=1$ if  $m =0$. As we will see, the quotient $\fN_{m,n}^l:=(\SU(2) \times \Sph^3)/ \Sph^1$ is diffeomorphic to $\Sph^3 \times \Sph^2$ if $m+n$ is even and diffeomorphic to $\Sph^3 \, \wt{\times} \, \Sph^2$ otherwise.
Consider~$\fN_{m,n}^l$ as an $\SU(2)$-manifold by defining 
$$g\cdot [(p,(z,w))] = [(gp,(z,w))],$$
for $g \in \SU(2)$. This action has isotropy groups isomorphic to $\Z_m$, $\Z_n$ and $\Z_{\gcd(m,n)}$ if $m$ and $n$ are both positive,~$\Z_n$ and $\SO(2)$ if $n > m =0$ and only one isotropy type $(\SO(2))$ if $m=n=0$. 
If $\gcd(m,n)$ is even, the action has ineffective kernel $\Z_2$ and hence it is an effective action by $\SO(3)$. 
Notice that the actions on $\fN_{1,1}^1$ and~$\fN_{2,2}^1$ are free, the actions on $\fN_{1,1}^1$, $\fN_{0,2}^1$ and $\fN_{0,0}^1$ are linear, and that, to complete the list of all linear actions on $\Sph^3 \times \Sph^2$, we should include the $\SO(3)$-action induced by the natural inclusion of $\SO(3) \subset \SO(4)$ acting on the first factor. 

We point out that $\fN_{m,n}^ l$ can also be described as a biquotient in the form 
$$ \{e\} \times \Delta\Sph^3  \setminus \, \Sph^3 \times \Sph^3 \times \Sph^3 /\,  \Delta \Sph^1$$
with $\Sph^1$ embedded as a circle of slope $(l,n - m,-n-m)$ in the maximal 3-torus, and the $\SU(2)$-action on the left in the first factor. Finally, observe that by O'Neill's formula the standard product metric on $\Sph^3 \times \Sph^3$ induces an invariant metric of nonnegative curvature on $\fN_{m,n}^l$.

Throughout this work, unless otherwise stated, $G$ will denote $\SO(3)$ or $\SU(2)$ and $M$ a simply-connected compact $G$-manifold of dimension 5. Our first result is a complete classification of all such  nonnegatively curved $G$-manifolds. 

 \begin{main}\label{thm metric} If the  $G$-manifold $M$ admits an invariant metric with nonnegative curvature, then it is equivariantly diffeomorphic to either $\Sph^5$, $\Sph^3 \times \Sph^2$, or $\fW = \SU(3)/\SO(3)$ with their natural linear $G$-actions, or $\fN_{m,n}^l$.
 \end{main}

Notice that the natural metrics on these manifolds are invariant by the $G$-action and have nonnegative curvature and that the manifolds can all be written as biquotients. 

For positive curvature we have the following partial classification.
\begin{main}\label{thm positive} If the $G$-manifold $M$ admits an invariant metric with  positive curvature, then it is either equivariantly diffeomorphic to~$\Sph^5$ with a linear action, or possibly~$\fW$ with the linear~$\SU(2)$-action, or~$\fN_{m,n}^l$ with trivial principal isotropy group, i.e., $\gcd(m,n)=1$ or 2.
\end{main}

It is natural to conjecture that in the context of \tref{thm positive} only the linear actions on $\Sph^5$ admit invariant metrics of positive curvature. In any case a complete classification in \tref{thm positive} would answer the generalized Hopf conjecture for $\Sph^3 \times \Sph^2$ with non-abelian symmetry.

Theorems \ref{thm metric} and \ref{thm positive} will be a consequence of a general equivariant classification of $\SO(3)$ and $\SU(2)$-actions in dimension five, a result that is of interest in its own right. We begin with the case without singular orbits.
\begin{main}\label{thm exceptional SU(2)}  If the $G$-manifold $M$ does not have singular orbits, then it is equivariantly diffeomorphic to either $\fN_{0,0}^1$ or $\fN_{m,n}^l$ for some choice of positive integers $m$, $n$ and $l$. The $\SO(3)$-manifolds correspond to~$\fN_{m,n}^l$ with~$m$ and $n$ even.  
\end{main}
We will see that these actions are pairwise non-equivalent for  different choices of the parameter~$l$ when $\gcd(m,n) \geq 3$ by showing that the fundamental group of the fixed point set of the principal isotropy group is isomorphic to $\Z_l$. For $\gcd(m,n)=1$ or 2, the actions on $\fN_{m,n}^l$ and $\fN_{m,n}^{l'}$ are equivalent precisely when $l = l'$ modulo $mn/\gcd(m,n)$.

\tref{thm exceptional SU(2)} easily implies the following.
\begin{cor} \label{cor T3} A $G$-action without singular orbits on $M$ extends  to $\U(2)\times \Sph^1$ if $G=\SU(2)$ and to $\SO(3) \times T^2$ if \mbox{$G=\SO(3)$}. In particular, $M$ admits an effective $T^3$-action.
\end{cor}

Finally, for actions with singular orbits we have

\begin{main} \label{thm main intro} If the $G$-manifold $M$ has singular orbits, then $M$ is equivariantly diffeomorphic to $\Sph^5$ with a linear action or:
\begin{enumerate}[(a)]
\item $G = \SO(3)$ and $M$ is equivalent to either $\fN_{0,2m}^1 = \Sph^3 \times \Sph^2$, connected sums of copies of the Wu-manifold and copies of the Brieskorn variety of type $(2,3,3,3)$, or connected sums of copies of~$\Sph^3 \times \Sph^2$ with the linear action by $\SO(3) \subset \SO(4)$ on the first factor;
\item $G = \SU(2)$ and $M$ is equivalent to either $\fN_{0,2m+1}^1 = \Sph^3 \wt{\times} \Sph^2$, or the Wu-manifold with the left action by \mbox{$\SU(2) \subset \SU(3)$}.
\end{enumerate}
\end{main}

The case of $G=\SO(3)$ was studied in \cite{Hudson1977}, although the author missed the equivariant connected sums of $\Sph^3 \times \Sph^2$ with the above $\SO(3)$-action, and did not describe some of the actions explicitly. For partial results about differentiable classifications, see  \cite{Nakanishi}, \cite{Oh1}, \cite{Oh2} and~\cite{Oike}.  In \cite{kleiner_thesis} the $\SO(3)$ and $\SU(2)$-actions on compact simply-connected 4-manifolds are classified.

This paper is organized as follows. In \sref{examples} we discuss preliminaries and describe the basic examples. In \sref{sec main} we introduce the $\SU(2)$-manifolds $\fN_{m,n}^l$, prove \cref{cor T3} assuming \tref{thm exceptional SU(2)} and prove some  results needed for the proof of \tref{thm exceptional SU(2)}. Sections~\ref{sec exceptional} and~\ref{sec singular} are devoted to the proofs of Theorems~\ref{thm exceptional SU(2)} and~\ref{thm main intro}. In~\sref{sec positive} we prove Theorems \ref{thm metric} and \ref{thm positive}.
\vspace{0.3cm}

\noindent {\bf Acknowledgments.} This work is part of the author's PhD Thesis at IMPA. He would like to express gratitude to IMPA for its hospitality and to his doctoral supervisors Luis A. Florit and Wolfgang Ziller for long hours of helpful and pleasant conversations during the preparation of this~paper.

\section{Preliminaries}\label{examples}

In this section we fix our notation and define some of the objects we will work with. Here $G$ is any compact Lie group not necessarily $\SO(3)$ or $\SU(2)$.

If $G$ acts on $M$ we denote by $G_p = \{g \in \G \, : \, gp = p\}$ the isotropy group of the point $p \in M$ and by~$(K)$ the conjugacy class of an isotropy group $K$ and call it {\it type of $K$}. The Principal Orbit Theorem guarantees that there exists one of smallest type, the {\it principal isotropy group} which will be denoted by $H$. By the Slice Theorem, a neighborhood of an orbit $Gp$ can be describe up to equivariant diffeomorphism by $G \times_K D$ where $K = G_p$ and $D$ is a disk in the normal space to $T_p(Gp)$. The linear action of $K$ on $D$ is called the {\it slice representation} at the point $p$.

If $H$ is a principal isotropy group then, dim$\,H \leq \text{dim}\,K$ for any isotropy subgroup $K$ of the action. If dim$\,H < \mbox{dim}\,K$ then we call $K$ a \textit{singular isotropy group}. 
If $\dim H = \dim K$ and $K$ has more connected components than $H$, then the group $K$ is called an \textit{exceptional isotropy group}. 

The dimension of the orbit space $M/G$ is called the  \textit{cohomogeneity} of the action. About the orbit space we will need the following basic result (see \cite{Bredon} p. 91, 190, 207 and 211).

\begin{prop}\label{prop quotient} If $M$ is simply-connected and $G$ is connected:
 \begin{enumerate}[(a)]
  \item The orbit space $M/G$ is simply-connected.
  \item \label{B1} If the action has cohomogeneity 2, then the orbit space is a topological surface with (or without) boundary. The boundary, if not empty, consists of the singular orbits and, in this case, there are no exceptional orbits.
  \item \label{B3} If the action has cohomogeneity 3, then $M/G$ is a simply-connected topological manifold possibly with boundary.
 \end{enumerate}
\end{prop}

Every nontrivial subgroup of $\SO(3)$ is isomorphic to either the cyclic group $\Z_k$, the dihedral group~$\D_m$, the tetrahedral group $\T$, the octahedral group $\O$, the icosahedral group $\I$, the circle~$\SO(2)$ or  the normalizer of $\SO(2)$ in $\SO(3)$, which is $\O(2)$ (see \cite{wolf} Section 7.1).

The special unitary group $\SU(2)$ can be seen as $\Sph^3 \subset \C^2$, or the unit quaternions, $\Sph^3 \subset \QH$. For $\ga$ and $\gb \in \C$ with $|\ga|^2 + |\gb|^2 = 1$ we have the following correspondence between the three expressions of an element in $\SU(2)$: 
\begin{eqnarray*}
\left(\begin{array}{cc} \ga & -\ovl{\gb} \\ \gb & \ovl{\ga} \end{array} \right)\in \C^{2\times 2} \quad \sim \quad  (\ga,\gb) \in \C^2 \quad \sim \quad \ga+\gb j \in \QH.
\end{eqnarray*}
The quaternion notation will be generally used for the group $\SU(2)$, while the $\Sph^3 \subset \C^2$ notation when considering $\SU(2)$ just as a manifold.  

Denoting by $\phi: \SU(2) \ra \SO(3)$ the 2-fold cover. The subgroups of $\SU(2)$ are isomorphic to $\Z_{2k+1}$ or the pre-images by $\phi$ of the subgroups of $\SO(3)$ (see~\cite{asoh} \textsection 2). Any closed nontrivial subgroup of $\SU(2)$ is then isomorphic to $\Z_k$, the dicyclic group $\Dic_n$, the binary tetrahedral group $\T^*$, the binary octahedral group $\O^*$, the binary icosahedral group $\I^*$, the circle or the $\Pin(2)$ group.

We introduce here some of the $G$-actions with $G=\SO(3)$ or $\SU(2)$ that will appear in our classification.

\begin{example}\label{ex sphere} {\it The linear $G$-actions on $\Sph^5$ are given by:}
\begin{enumerate}[(a)]
 \item \label{ex SO(3) on S5}$A \in \SO(3) \mapsto \diag(A,1,1,1)\in \SO(6)$ with isotropy groups $\SO(2)$ and $\SO(3)$.
 \item \label{ex linear 2OT} $A \in \SO(3) \mapsto \diag(A,A)\in \SO(6)$ with isotropy groups $\{\Id\}$ and $\SO(2)$.
\item \label{ex SU2 S5}$B \in \SU(2) \mapsto \diag(B,1)\in \SU(3)$ with isotropy $\{\Id\}$ and a circle of fixed points.
 \item \label{ex irreducible} The representation of $\SO(3)$ on $\R^6$ is defined by $A\cdot X=AXA^{-1}$ where $X \in \R^6$ is a 3 by 3 symmetric matrix. The action is by isometries on~$\R^6$ with the standard inner product. The isotropy groups are $\Z_2 \times \Z_2$, $\O(2)$ and $\SO(3)$ and the orbit space is a sector of angle $\pi/3$ in $\R^2$. Thus, it induces an $\SO(3)$-action in~$\Sph^5$ with the same isotropy groups and quotient a topological 2-disk with angle~$\pi/3$ in each of the two fixed points.
\end{enumerate}
\end{example}

\begin{example} \label{ex 1missed}{\it The linear $G$-actions on $\Sph^3 \times \Sph^2$}.\newline
They are given by the nonequivalent embeddings of $G$ in $\SO(4) \times \SO(3)$.
\begin{enumerate}[(a)]
 \item  \label{ex 1missed item}$\SO(3)$ in the first factor. The fixed point set is the union of two copies of 2-spheres, any other point has isotropy $\SO(2)$ and the orbit space is diffeomorphic to~$\Sph^2 \times [-1,1]$.
 \item \label{ex 1missed S1}$\SO(3)$ in the second factor. The unique isotropy type is $(\SO(2))$. It is equivalent to~$\fN_{0,0}^1$.
 \item \label{ex 1missed item Id S1}The diagonal inclusion of $\SO(3)$, that is $A \in \SO(3) \mapsto (\diag(A,1), A) \in \SO(4)\times \SO(3)$. The isotropy groups are $\{\Id\}$ and $\SO(2)$. The action is equivalent to $\fN_{0,2}^1$.
 \item \label{ex 1missed free}$\SU(2)$ in the first factor. The action of $\SU(2)$ on $\Sph^3$ is equivalent to the left multiplication of Lie groups, thus it is free with quotient $\Sph^2$. This action is equivalent to $\fN_{1,1}^0$. Note that this is equivalent to the linear action given by the diagonal inclusion of $\SU(2)$ in $\SO(4) \times \SO(3)$, by the classification of fiber bundles over spheres (see, Corollary 18.6 in \cite{steenrod}). 
\end{enumerate}
\end{example}

\begin{example}\label{ex SU(2) on W}{\it The $\SU(2)$-action on the Wu-manifold $\fW:=\SU(3)/\SO(3)$}. \newline
Given $B\in \SU(2)$ and $[C] \in \fW$, the action is  $B\cdot [C] = [(\diag(B,1)C)]$. The isotropy types are $\{\Id\}$ and $(\SO(2))$ and the quotient is a 2-disk.
\end{example}

\begin{example}\label{ex SO(3) on Wu}{\it The $\SO(3)$-action on $\fW$.} \newline
 The group $\SO(3)$ acts on $\fW$ as $A\cdot [B] = [AB]$.  The isotropy types are $(\Z_2 \times \Z_2)$, $(\O(2))$ and~$\SO(3)$. If the action is by isometries, the quotient is a flat triangle with vertices the fixed points and each edge corresponds to one of the three distinct embeddings of $\O(2)$ in $\SO(3)$. It was proved in \cite{Hudson1977} that this is the unique $\SO(3)$-action on the Wu-manifold up to conjugacy. 
\end{example}

\begin{example} \label{ex brieskorn} {\it The $\SO(3)$-action on the Brieskorn variety, $\fB$  of type $(2,3,3,3)$. }\newline
The Brieskorn variety of type $(2,3,3,3)$ can be defined as
$$\fB= \left\{(z_o,z_1,z_2,z_3) \in \C^4 \, ; \, z_o^2 + z_1^3 + z_2^3 + z_3^3 = 0 \text{ and } |z_o|^2 + |z_1|^2 + |z_2|^2 + |z_3|^2 = 1\right\}. $$
In \cite{Hudson1977} this action is constructed  as an example of an $\SO(3)$-manifolds with isotropy types $(\Z_2 \times \Z_2)$, $(\O(2))$ and four fixed points by a process of gluing four open sets and the manifold is identified computing topological invariants. As far as we know an explicit description of the action is not~known.
\end{example}

Given two $n$-dimensional $G$-manifolds with fixed points, choose Riemannian metrics invariant under the $G$-actions and consider a small ball of radius $r$ around a fixed point in each manifold. If the isotropy actions of $G$ on the slices of those fixed points are the same, then the actions on the boundaries of the balls are equivalent and we can form a connected sum of the two $G$-manifolds gluing along those spheres to, possibly, obtain a new $G$-manifold. Particular examples of this are:

\begin{example}\label{ex kmissed}{\it The $\SO(3)$-action on the connected sum of $k$ copies of $\Sph^3 \times \Sph^2$.}\newline
We start with the $\SO(3)$-action on $\Sph^3\times \Sph^2$ from \eref{ex 1missed} (\ref{ex 1missed item}). At a fixed point the isotropy representation is given by $\SO(3)$-action on $\R^3 \times \R^2$ that is standard in the first factor and trivial on the second. We can now take the connected sum of 2 copies of $\Sph^3 \times \Sph^2$ at the fixed points. This provides a connected sum of two fixed 2-spheres, so if we do this for $k$ copies of $\Sph^3 \times \Sph^2$, we obtain $k+1$ fixed 2-spheres. The orbit space of the action is diffeomorphic to a 3-sphere with $k+1$ three-disks removed.
\end{example}
The $\SO(3)$-manifolds in \eref{ex kmissed} were overlooked in the classification in \cite{Hudson1977}.

\begin{example}{\it The $\SO(3)$-actions on $k\, \fW \, \# \, l \, \fB$.} \label{ex hudson}\newline
The isotropy representation around an isolated fixed point of an $\SO(3)$-manifold of dimension~5 must be the unique irreducible one (see \eref{ex sphere} (\ref{ex irreducible})).
The $\SO(3)$-action on $\Sph^4$ where the connected sum takes place has quotient an interval, so there are exactly two ways to connect the manifolds. In~\cite{Hudson1977}, it is shown that depending on the way that $\fW$ and $\fB$ are connected we get distinct $\SO(3)$-manifolds.
\end{example}


\section{The main example}\label{sec main}

In this section we present the main class of examples of our classifications and prove some of its properties that will be used to obtain \tref{thm exceptional SU(2)}. The following construction is crucial since it generates all 5-dimensional $G$-manifolds without singular orbits and most actions with singular orbits for $G=\SO(3)$ or $\SU(2)$.

\begin{example}[Main example]\label{ex except} Let $m$, $n$ and $l$ be nonnegative integers, we assume that $m \leq n$ and set $l=1$ whenever $m=0$. Consider the $\Sph^1$-action on $\SU(2) \times \Sph^3$ given by 
\begin{equation}\label{eqtheta} x * (p,(z,w)) = (px^l,(x^mz,x^nw)),\end{equation}
where $p\in \SU(2)$ and $(z,w) \in \Sph^3 \subset \C^2$. We regard the $\SU(2)$ factor as the group of unit quaternions. Take $\Sph^1 = \{e^{i\gt}\}$ to the inclusion of the circle in both $\C$ and $\QH$. The quotient 
$$\fN_{m,n}^l = \SU(2) \times_{\Sph^1}\Sph^3$$
 is a manifold whenever $\gcd(l,m)=\gcd(l,n)=1$. It is clear from the sequence of homotopies that the manifold $\fN_{m,n}^l$ is simply-connected.
\end{example}
\begin{prop}\label{prop S3 bdls} The manifolds $\fN_{m,n}^l$ are diffeomorphic to $\Sph^3 \times \Sph^2$ if $m+n$ is even, otherwise they are diffeomorphic to the nontrivial $\Sph^3$-bundle over $\Sph^2$, denoted by  $\Sph^3 \, \wt{\times} \, \Sph^2$.
\end{prop}
\begin{proof} 
The sequence of homotopies of the principal bundle $\Sph^1 \ra \Sph^3 \times \Sph^3 \ra \fN_{m,n}^l$ and Hurewicz isomorphism guarantee that $H_2(\fN_{m,n}^l) \simeq \pi_2(\fN_{m,n}^l) \simeq \Z$. By \cite{jason} p. 77 the second Stiefel-Whitney class $w_2=0$ if $m+n$ is even and $w_2 = 1$ otherwise. So, the result follows from Barden-Smale Theorems, c.f. \cite{Barden1965}
and~\cite{Smale1962}.
\end{proof}

Now we define the $\SU(2)$-action on $\fN_{m,n}^l$ by 
$$g \cdot [(p,(z,w))] = [(gp,(z,w))].$$
We will also denote this $\SU(2)$-manifold by $\fN_{m,n}^l$.
This action has the same isotropy structure as the $\Sph^1$-action on the second factor $\Sph^3$. If the integers $l$, $m$ and $n$ are nonzero, the isotropy group of the point $[(p,(z,w))]$ is respectively $\Z_m$, $\Z_n$ or $\Z_{\gcd(m,n)}$ according $w=0$, $z=0$ or both~$z$ and $w$ are nonzero. For convenience we set $\gcd(0,0) = 1$.

In general, for $\SU(2)$-actions, if the \pig contains $\Z_2$ as a subgroup, the action is ineffective since $\Z_2$ is normal in $\SU(2)$. Therefore, $\fN^l_{2m,2n}$  becomes  an effective \mbox{$\SO(3)$-m}anifold with isotropy groups $\Z_{m}$,  $\Z_{n}$ and $\Z_{\gcd(m,n)}$. 
Observe that in this case the underlying manifold is  diffeomorphic to $\Sph^3 \times \Sph^2$ and $l$ has to be odd.

\begin{rem}
The $\Sph^1$-action \eqref{eqtheta} is a restriction of the $\SU(2) \times T^3$-action on $(\SU(2) \times \Sph^3)$ given~by
$$(g, (r,s,t)) \cdot (p,(z,w)) = (gpr, (sz, tw))$$
where $r,s,t \in \Sph^1$. In our example $\{1\} \times \Delta \Sph^1 = \{(1,(x^l, x^m, x^n))\} \subset \SU(2) \times T^3$ is the $\Sph^1$ group that acts on $\SU(2) \times \Sph^3$. Hence, the quotient $\fN_{m,n}^l$   admits an action by  $\SU(2) \times  (T^3 /\Delta \Sph^1)$ with kernel generated by $(-1,[(-1,1,1)])$. Notice that if $m$ and $n$ are even then \mbox{$(1,(-1,1,1))\in \Delta \Sph^1$}, so the kernel is $\Z_2 \subset \SU(2)$ and $\SO(3) \times T^2$ acts effectively on $\fN_{m,n}^l$. Otherwise, the ineffective kernel is the diagonal $\Z_2$ in $\SU(2) \times \Sph^1 \subset \SU(2) \times T^2$, thus  $\U(2) \times \Sph^1$ acts effectively on~$\fN_{m,n}^l$. This proves \cref{cor T3}, assuming  \tref{thm exceptional SU(2)}.
\end{rem}

As we pointed out, if $lmn \neq 0$, the isotropy types of $\fN_{m,n}^l$ are  $(\Z_m)$, $(\Z_n)$ and $(\Z_{\gcd(m,n)})$. Hence, if $m = n$, the action has only one isotropy type $(\Z_n)$. In particular, we get free actions on $\Sph^3 \times \Sph^2$ when $m$ and $n$ are both equal to either 1 or 2. 
The $\SU(2)$-manifolds~$\fN_{1,1}^l$ are all equivalent since there is only one isomorphism class of $\SU(2)$-principal bundles over~$\Sph^2$ (c.f. Corollary~18.6~in~\cite{steenrod}).
The same result shows that there are two non-equivalent \mbox{$\SO(3)$-principal} bundles over~$\Sph^2$, but just one of them is simply-connected, $\fN_{2,2}^1$, with the free $\SO(3)$-action. 
Notice that the free $\SU(2)$-manifold~$\fN_{1,1}^0$ corresponds to the left multiplication on the first factor of $\SU(2) \times \Sph^2$.

If $m=n=0$ then $l=1$ and the $\Sph^1$-action on the first factor reduces to the Hopf action with quotient diffeomorphic to $\Sph^2 \simeq \SU(2)/\SO(2)$. 
Thus the $\SU(2)$-action is the natural product on the cosets and has unique isotropy type equal to $(\SO(2))$. On the other hand, if $m = 0 < n$, then $l=1$ again and the isotropy types are $(\Z_n)$ and $(\SO(2))$.

\begin{rem} \label{rem sign of l} We do not obtain new $G$-manifolds by taking negative integer parameters. In fact, if we regard the $\Sph^1$-action on the first factor  of $\SU(2) \times \Sph^3$ considering $\SU(2) \subset \C^2$ rather than the unit quaternions, for $x \in \Sph^1$ and $(u,v) \in \SU(2)$  we obtain the action \mbox{$x\cdot (u,v) = (ux^l, v \ovl{x}^l)$}. Therefore, the $\SU(2)$-manifolds $\fN_{m,n}^l$ and $\fN_{m,n}^{-l}$ are equivalent by switching $(u,v)$ to $(v,u)$. In the same way we can consider the $\SU(2)$-equivariant diffeomorphism \mbox{$f: \fN_{m,n}^l \ra \fN_{-m,n}^l$}  which takes $[(p,(z,w))]$ to $[(p,(\ovl{z},w))]$. The equivalence for $n$ negative is analogous.
 \end{rem}

Hereafter in this section $\fN_{m,n}^l$ will be denoted by $\fN_{n_1,n_2}^l$ and its elements are now represented by $[(p,(z_1,z_2))]$. Now we will prove some results that will be crucial to obtain \tref{thm exceptional SU(2)}. \pref{prop invariant} and \lref{lem clutching} describe respectively the slice representations in the neighborhoods of the exceptional orbits of $\fN_{n_1,n_2}^l$ and the clutching function of the decomposition into two equivariant neighborhoods of the exceptional orbits. The slice representations and the homotopy class of the clutching function arise as invariants to be used in the proof of \tref{thm exceptional SU(2)}. 

\begin{prop} \label{prop invariant} The $\SU(2)$-manifold $\fN_{n_1,n_2}^l$ with $n_1$ and $n_2$ positive, when written as the union of the slice representations has the form
$$\fN_{n_1,n_2}^l = \SU(2)\times_{\Z_{n_1}} D^2 \, \bigcup_{\varphi} \, \SU(2)\times_{\Z_{n_2}} D^2,$$
where $\Z_{n_j}$ acts on $\SU(2) \times D^2$ by $\xi\cdot(p,z)=( p\xi  , \xi^{n_il^{-1}}z)$ and $l^{-1}$ is the inverse of $l$ in $\Z_{n_j}$ for $1 \leq i \neq j \leq 2$. 
\end{prop}
\begin{proof}
First consider $\Sph^3 = \B_1 \cup \B_2$ where $\B_j = \{(z_1,z_2)\in \Sph^3 : |z_j|\geq 1/\sqrt{2}\}$ for $j=1$ or 2, and the identification of the boundaries is the trivial one (the identity map).  Note that 
$$\fN_{n_1,n_2}^l = \SU(2) \times_{\Sph^1} \B_1 \,\bigcup_{\Id}\, \SU(2) \times_{\Sph^1} \B_2.$$
Now we describe an equivalence between $\SU(2) \times_{\Sph^1} \B_{j}$ and a certain quotient of $\SU(2)\times D^2$ by~$\Z_{n_j}$. Assume $j=2$, the other case being analogous.
For $[(p,(z_1,z_2))] \in \SU(2)\times_{\Sph^1}\B_2$ we have $z_2 \neq  0$ therefore we can write $[(p,(z_1,z_2))] = [(\, p,(\, z_1, |z_2|\, z_2/|z_2| \, ))]$. Take $x = \zeta\eta_2 \in \Sph^1$ with $\eta_2^{n_2}=1$ and $\zeta^{n_2} = \ovl{z_2}/|z_2|$ with $\arg(\zeta) < 2\pi/n_2$ in order to obtain $x^{n_2} = \ovl{z_2}/|z_2|$. Then
 $$
[(\, p,(\, z_1, z_2 \, ))] = [(\, p\zeta^l\eta_2^l,(\, \eta_2^{n_1}\zeta^{n_1}z_1, |z_2| \, ))] =  
[(\, \hat{p}\eta_2^l,(\, \eta_2^{n_1} \hat{z_1}\,, \,\sqrt{1-|\hat{z_1}|^2}\, ))],$$
where $\eta_2 \in \Z_{n_2} \subset \Sph^1 \subset \C$, we called $\hat{p} = p\zeta^l$ for each $p \in \SU(2)$ and $\hat{z_1} = \zeta^{n_1} z_1$. So, for some equivariant diffeomorphism $\varphi: \SU(2) \times_{\Z_{n_1}} \Sph^1 \ra \SU(2) \times_{\Z_{n_2}} \Sph^1$ we have
\begin{equation}\label{eqn broken N}\fN_{n_1,n_2}^l = \SU(2)\times_{\Z_{n_1}} D^2 \,\bigcup_{\varphi} \, \SU(2)\times_{\Z_{n_2}} D^2,\end{equation}
with the actions $\Z_{n_j} \hra \SU(2) \times D^2$ given by $\eta_j\cdot (p,z) = ( p \eta_j^l, \eta_j^{n_i}z)$. The result follows from considering the generator $\xi_j = \eta^l_j$ of $\Z_{n_j}$.
\end{proof}
 
Denote $\gcd(n_1, n_2) = d$ and $n_j = dq_j$. Notice that the principal isotropy group $\Z_d \subset \Z_{n_j}$ acts trivially on the slice $D^2$, so $\SU(2) \times_{\Z_{n_j}}D^2$ is equivalent to $\SU(2)/\Z_d \times_{\Z_{q_j}} D^2$ and the equivalence is given by  \mbox{$[(p,z)]  \sim [(p\Z_d,z)]$}. Therefore the clutching function $\varphi$ is an equivariant map defined from \mbox{$\SU(2)/\Z_d \times_{\Z_{q_1}} \Sph^1$} to   \mbox{$\SU(2)/\Z_d \times_{\Z_{q_2}} \Sph^1$}. It is thus sufficient to compute $\varphi$ along the path $t \mapsto [(\Z_d, \mu(t)^{n_2})]$ where 
$$\mu(t)=e^{2\pi it/dq_1q_2} \in \C \subset \QH.$$
Whenever necessary to make arguments more clear, we will use $\lb .\,,. \rb$  for the classes in \mbox{$G/\Z_d \times_{\Z_{q_2}} D^2$}.

\begin{lem} \label{lem clutching} The clutching function $\varphi:\SU(2)/\Z_d \times_{\Z_{q_1}} \Sph^1 \ra \SU(2)/\Z_d \times_{\Z_{q_2}} \Sph^1$ is
\begin{equation} \varphi([(\Z_d, \mu(t)^{n_2})]) = \lb (\ovl{\mu(t)}^l \Z_d, \ovl{\mu(t)}^{n_1})\rb. \end{equation}
\end{lem}
\begin{proof}
Consider the notation introduced in the proof of \pref{prop invariant}. First, identify the elements \mbox{$[(\Z_d, \mu(t)^{n_2})] \in \SU(2)/\Z_d \times_{\Z_{q_1}} \Sph^1$} with $[(1, (1/\sqrt{2} , \mu(t)^{n_2}/\sqrt{2})]$ that lie in  the boundary of $\SU(2) \times_{\Sph^1} B_1$. Let $x(t) = \ovl{\mu(t)}\eta_2 \in\Sph^1$, recall that \mbox{$\eta_2^{n_2} = 1$}, and use the equivalence by $\Sph^1$ to get $[(1, (1/\sqrt{2} , \mu(t)^{n_2}/\sqrt{2})] = [( \ovl{\mu(t)}^l\eta_2^l, ( \ovl{\mu(t)}^{n_1}\eta_2^{n_1}/\sqrt{2},1/\sqrt{2}))]$. If we see this last element in \mbox{$\SU(2)\times_{\Sph^1} B_2$}, then the previous identification yields $\lb (\ovl{\mu(t)}^l \Z_d, \ovl{\mu(t)}^{n_1})\rb \in \SU(2)/\Z_d \times_{\Z_{q_2}} D^2$.
\end{proof}

\pref{prop invariant} implies that if  $\fN_{n_1,n_2}^{l}$ and $\fN_{n_1,n_2}^{l'}$ are equivalent, then $l' \equiv \pm l \mod q_1q_2$. Indeed, we will see in \tref{lem equivariance} that it can be improved to $l' \equiv \pm l \mod 2q_1q_2$ if $d = 2$ and for $d \geq 3$, the next result guarantees that the parameter $l\geq 0$ itself is an invariant of the action.

\begin{prop}\label{claim fixed} If the principal isotropy group $H$ of $\fN_{n_1,n_2}^l$ is isomorphic to $\Z_d$  for $d \geq 3$, then the fixed point set $(\fN_{n_1,n_2}^l)^H$ is a disjoint union of two copies of lens spaces $\Sph^3/\Z_l$.
 \end{prop}
\begin{proof} Let $H \simeq \Z_d$ be the subgroup of $\SU(2)$ generated by $e^{2\pi i/d}$ and notice that $N(H) = N(\Sph^1)$, where $H \subset \Sph^1$ and $N(K)$ is the normalizer of the subgroup $K \subset G$ in $G$. This easily implies that an element $[(p,(z_1,z_2))]$ belongs to $M^H$ if, and only if, $p \in N(H)$. Thus $(\fN_{n_1,n_2}^l)^H= N(H)\times_{\Sph^1}\Sph^3$. Therefore
$$(\fN_{n_1,n_2}^l)^H= (\Sph^1 \times_{\Sph^1} \Sph^3) \bigcup (\Sph^1 \times_{\Sph^1} \Sph^3),$$
since $N(\Z_d) \simeq \Pin(2)$.

Notice that every $[(y,(z_1,z_2))] \in \Sph^1 \times_{\Sph^1}\Sph^3$ has a representative with $y=1$. 
In fact, $[(y,(z_1,z_2))] = [(1,(\xi^{n_1}\zeta^{n_1} z_1,\xi^{n_2}\zeta^{n_2} z_2))]$ where $\zeta^l = y$ with $\arg(\zeta)<2\pi/l$ and $\xi^l=1$. Now define $\hat{z_1} =\zeta^{n_1} z_1$ and $\hat{z_2} =\zeta^{n_2} z_2$, as a new parametrization for the 3-sphere.  Thus $\Sph^1 \times_{\Sph^1} \Sph^3$ is diffeomorphic to the quotient of $\Sph^3$ by the $\Z_l$-action $\xi\cdot(z_1,z_2) = (\xi^{n_1} z_1,\xi^{n_2} z_2).$
Hence $\Sph^1 \times_{\Sph^1} \Sph^3$ is a Lens space $\Sph^3/\Z_l$. 
\end{proof}

\begin{rem} Clearly \pref{claim fixed} only has an assumption on the principal isotropy group, so the fixed point set $M^H$ when $M$ is either $\fN_{0,0}^1$ or $\fN_{0,n}^1$ is the disjoint union of two copies of a 3-sphere.
\end{rem}

\section{Actions without singular orbits and the proof of Theorem C} \label{sec exceptional}
For the proof of \tref{thm exceptional SU(2)} we first verify that the quotient is homeomorphic to a two or three-dimensional sphere. Then we show that, if the action has only one isotropy type, the circle and cyclic groups are  the only possible isotropy subgroups of the action and classify the actions with exactly one isotropy type $(\SO(2))$. 
We conclude that the action has at most two exceptional orbits and that the pair $(H,K)$ of principal and exceptional isotropy groups must be $(\Z_d, \Z_n)$, $(\D_2, \T)$ or $(\Dic_2, \T^*)$. Then, we use the Slice Theorem to construct $M$ as a union of two neighborhoods of the exceptional orbits and compute the fundamental group of the union to see that only cyclic isotropies can occur for simply-connected $G$-manifolds.
The actions constructed in this way depend on three integer parameters: one comes from the clutching function,  while the other two correspond to the isotropy representations around the exceptional orbits.
We finish the proof by establishing a one-to-one correspondence between the distinct general actions constructed and the $\SU(2)$-manifolds~$\fN_{n_1,n_2}^l$ for $n_1n_2 \neq 0$.  

Let us understand the orbit space.
\begin{lem}\label{lem ex qt} Let $G \hra M^5$ be an action without singular orbits. Then either it has only one isotropy type $(\SO(2))$ and $M/G \simeq \Sph^3$, or the isotropy groups are finite and $M/G \simeq \Sph^2$.
\end{lem}
\begin{proof}
 By Proposition \ref{prop quotient}, the orbit space is homeomorphic to a compact, simply-connected 2 or 3-dimensional  manifold, possibly with boundary. If the action has only one isotropy type, the quotient is a base space of a fiber bundle, thus it is a simply-connected topological two or three manifold without boundary. In any case $M/G$ is a topological sphere. If the action has exceptional orbits the quotient is two dimensional without boundary. In fact, if the principal orbits have codimension 3, the principal and exceptional isotropy groups have to be $H=\SO(2)$ and $K=\O(2)$ (or $K = \Pin(2)$ if $G=\SU(2)$). Then the orbits are $\Sph^2$ and~$\RP^2$.
 But the orbits of maximal dimension are orientable (see \cite{Bredon} p. 188), so there is no cohomogeneity 3 action with exceptional orbits.
On the other hand, in the cohomogeneity two case, by \pref{prop quotient} (\ref{B1}), the quotient does not have boundary, so $M/G \simeq \Sph^2$ topologically. 
\end{proof}

\subsection{Actions with only one (noncyclic) isotropy type}\label{Sec one SO(3)}
In this section we show that a simply-connected $G$-manifold with only one orbit type has isotropy either $\Z_m$ or $\SO(2)$, classifying the actions in the second case. The classification of the actions with only one isotropy type $(\Z_m)$ will be done later together with the classification of $G$-manifolds with exceptional orbits.  

For general $G$-actions with only one isotropy type we have the following:
 
 \begin{prop}\label{lem counting UOT} There are as many equivalence classes of $G$-manifolds with only one isotropy type $(H)$ and orbit space $M/G \simeq \Sph^n$ as elements in $\pi_{n-1}(\gC_H)$.
\end{prop}
\begin{proof}
It is known (c.f. A. Borel \cite{borel}) that given a closed subgroup $H \subset G$, there is a bijective correspondence between the set of isomorphism classes of principal bundles $\gC_H \ra P \ra B$, where $\gC_H=N(H)/H$, and the set of equivalence classes of $G$-manifolds $M$ with unique orbit type $(G/H)$. 
The set of isomorphism classes of $F$-bundles over $\Sph^n$ is in bijection with $\pi_{n-1}(F)$ whenever $F$ is arcwise connected (c.f.~\cite{steenrod}, Corollary~18.6). Now, the proposition follows from the fact that for a Lie group $K$, the principal $K$-bundle is determined by the $K_{o}$-fiber bundle over the same basis, where $K_o$  is the identity component of~$K$. \end{proof}

By \pref{lem counting UOT} the number of $G$-manifolds with only one orbit type equal to $(G/H)$ and quotient an $n$-sphere is the order of the $(n-1)$-th homotopy group of $N(H)/H$, where $N(H)$ is the normalizer of $H$ in $G$. These homotopy groups are presented in Table \ref{tab SOUOT} (see~\cite{Hudson1977} for $\SO(3)$ and \cite{asoh} for $\SU(2)$), for $G=\SO(3)$ or $\SU(2)$.
\hspace{-0,5cm}\begin{table}[h]
\begin{tabular}{l|cccccccccc}
\multicolumn{11}{c}{{\bf G = SO(3)}}\\
{\bf H} &   $\{1\}$ & $\Z_2$ & $\Z_m$        & $\D_2$    & $\D_m$   & $\T$     &   $\I$   &   $\O$   & $\SO(2)$  & $\O(2)$ \\
{\bf N(H)}&  $\SO(3)$  &  $\O(2)$ &  $\O(2)$       &  $\O$       & $\D_{2m}$ & $\O$     &   $\I$   &  $\O$   &  $\O(2)$  & $\O(2)$  \\
{\bf N(H)/H}& $\SO(3)$ & $\SO(2)$  & $\O(2)$  & $\D_3$    &  $\Z_2$  & $\Z_2$&$\{1\}$& $\{1\}$& $\Z_2$& $\{1\}$\\
{\bf $\pi_{n-1}$(N(H)/H)}& $\Z_2$ & $\Z$  & $\Z $  & $\{1\}$    &  $\{1\}$  & $\{1\}$ & $\{1\}$& $\{1\}$& $\{1\}$& $\{1\}$\\
\multicolumn{11}{c}{}\\
\multicolumn{11}{c}{{\bf G = SU(2)}}\\
{\bf H}&$\{1\}$ &$\Z_2$ &$\Z_m$& $\Dic_2$ &$\Dic_m$ & $\T^*$ & $\I^*$ & $\O^*$ & $\SO(2)$            & $\Pin(2)$\\
{\bf N(H)}&$\SU(2)$& $\SU(2)$ &  $\Pin(2)$  & $\O^*$  & $\Dic_{2m}$        & $\O^*$          & $\I^*$         &    $\O^* $    & $\Pin(2)$           & $\Pin(2)$\\
{\bf N(H)/H}&$\SU(2)$& $\SO(3)$ & $\Pin(2)$ & $\D_3$  & $\Z_2$        &    $\Z_2$       &   $\{1\}$   &   $ \{1\}$  & $\Z_2$& $\{1\}$\\
{\bf $\pi_{n-1}$(N(H)/H)}& $\{1\}$ & $\Z_2$  & $\Z $  & $\{1\}$ & $\{1\}$  &  $\{1\}$ &  $\{1\}$ &  $\{1\}$ & $\{1\}$ & $\{1\}$
\end{tabular}
\caption{{\small Here $m\geq 3$ and $n$ is the cohomogeneity of the action, i.e. $n=2 + \dim H$}}
\label{tab SOUOT}
\end{table}

The following are the simplest examples.
\begin{example}\label{ex trivial} Let $H$ be a Lie subgroup of $G$ and $X$ be a manifold. Define a $G$-action on $(G/H)\times X$ by $g\cdot (kH,x) = (gkH,x)$. This action has unique isotropy type $(H)$ and its orbit space is $X$. 
\end{example}

From Table \ref{tab SOUOT} we see that there are two distinct free $G$-manifolds for $G=\SO(3)$ but only one if $G=\SU(2)$. One of the two free $\SO(3)$-manifolds is $\SO(3) \times \Sph^2$ which is not simply-connected, the other is $\fN_{2,2}^1$, hence simply-connected. The free $\SU(2)$-action is $\fN_{1,1}^0$, see \eref{ex 1missed} (\ref{ex 1missed free}).

The $\SO(3)$-action with unique isotropy type $(\SO(2))$ is $\fN_{0,0}^1$. It is the linear $\SO(3)$-action in \eref{ex 1missed} (\ref{ex 1missed S1}). Notice that $\SU(2)$ is a rank one Lie group whose center is~$\Z_2$, thus all the circle subgroups of $\SU(2)$ contain $\Z_2$, hence the $\SU(2)$-action with unique isotropy $(\SO(2))$ also corresponds to $\fN_{0,0}^1$.
 Finally, all the \mbox{$G$-manifolds} with unique isotropy $H$ equal to $\D_m$ with $m \geq 2$, $\T$, $\I$, $\O$ or $\O(2)$, if $G=\SO(3)$ and for $H$ isomorphic to $\Dic_m$ with $m \geq 2$, $\T^*$, $\I^*$, $\O^*$ or $\Pin(2)$ if $G=\SU(2)$ are described by \eref{ex trivial}. But none of them is simply-connected since in all these cases the fundamental group of $G/H$ is nontrivial.

For each $m\geq 3$ there are infinitely many examples of $G$-manifolds with unique isotropy type~$(\Z_m)$. The same holds for $m=2$ and $G=\SO(3)$, but for $G=\SU(2)$ there are exactly two such actions, they are ineffective and coincide with the free $\SO(3)$-manifolds. We will see in the next section that~$\fN_{m,m}^l$ are precisely the examples with isotropy~$\Z_m$.

\subsection{Actions with exceptional orbits or unique cyclic isotropy type}\label{subsec exceptional}
In this section we conclude the proof of Theorem \ref{thm exceptional SU(2)}. We will  classify the $G$-manifolds with exceptional orbits and actions with only one isotropy type $(\Z_m)$. The latter can be seen as a particular case of the former when the principal and exceptional isotropy groups coincide. The condition of simply-connectedness imposes strong restrictions on the isotropies (c.f., \pref{lem exceptional}) and limits the number of exceptional orbits to two (c.f., \lref{lem number except orbits}). In this situation we can construct $M$ as a union of the neighborhoods of the exceptional orbits using the Slice Theorem.

\begin{prop} \label{lem exceptional} If the $G$-manifold $M$ has exceptional orbits, then the pair of principal and exceptional isotropy types, $(H,K)$, is either $(\Z_{d},\Z_m)$, $(\D_2, \T)$ or  $(\Dic_2,  \T^*)$. Moreover, exceptional orbits are isolated and $H$ is the kernel of the slice representation of~$K$.
\end{prop}
In order to prove \pref{lem exceptional} the following will be essential.
\begin{lem} \label{lem except isotropy} Let $K \subset G$ be a finite subgroup with a normal subgroup $N \triangleleft K$ such that $K/N \simeq \Z_n$ for some $n \geq 3$. Then the pair $(N , K)$ is either $(\Z_{d},\Z_m)$, $(\D_2, \T)$ or  $(\Dic_2,  \T^*)$.
\end{lem}
The proof of the lemma is a case-by-case analysis on the finite subgroups of $\SO(3)$ and $\SU(2)$ and will be omitted. We refer to \cite{coxeter} and \cite{wolf} Section 7.1 for details about these groups.

\begin{proof}[Proof of \pref{lem exceptional}] 
As observed in \lref{lem ex qt} the isotropy groups are finite. Let $K \subset G$ be an exceptional isotropy group and consider the slice representation $\rho: K \ra \O(2)$ of $K$ with $N = \ker(\rho)$. The quotient $\R^2/(K/N) \simeq \R^2$ since $M/G$ is homeomorphic to $\Sph^2$, therefore $K/N \simeq \Z_n \subset \SO(2)$ with $n\geq 3$. So, the pairs $(N,K)$ are determined by \lref{lem except isotropy}.

We claim that $H=N$. This is clear if $K$ is cyclic. If it is not cyclic then $N$ is an index three subgroup of $K$ and since $N \subset H \varsubsetneq K$, we get $H=N$. The slice action of $K$ on $\R^2$, represented~by~$\rho$, only fixes the origin, thus the exceptional orbits are isolated.
\end{proof}

The action of the tetrahedral group and the binary tetrahedral group on the linear slice~$\R^2$ have kernel $\D_2$ and $\Dic_4$ respectively, since in our cases the principal isotropy groups are normal subgroups of $K$. So, the effective actions to be considered on $\R^2$ in theses cases are by $\Z_3\simeq\T/\D_2 = \T^*/\Dic_2$.

\begin{lem}\label{lem number except orbits} There are at most two exceptional orbits.
\end{lem}
\begin{proof}
The exceptional orbits in the quotient $M/G \simeq \Sph^2$ (topologically) represent orbifold singularities, in fact, a neighborhood of an exceptional orbit is parametrized by $\R^2/\Z_m$ with $m \geq 3$. It is known that a 2-dimensional orbifold with underlying space $\Sph^2$ with more than two singularities has nontrivial orbifold fundamental group (c.f. Thurston \cite{thurston}, p. 304). On the other hand, in our situation, there is an onto map from $\pi_1(M)$ to the orbifold fundamental group of the quotient $M/G$ (c.f. Molino \cite{molino}, p. 273 and 274). Therefore, there are at most two exceptional orbits since~$M$ is simply-connected.
\end{proof}







 A neighborhood of an exceptional orbit $G/K$ is given by $A = G \times_{K} D^2$, where $D^2 \subset \C$ is the 2-disk and the linear slice action of $K$ on $D^2$ has kernel equal to the principal isotropy group $H \subset K$ of the $G$-action. The action of $G$ on $A$ is given by $g\cdot[(p,z)] = [(gp,z)]$. The slice action only fixes the origin in $D^2$, since the exceptional orbits are isolated.  The manifold $M$ with at most two exceptional orbits can be written as 
\begin{eqnarray}\label{eqn glue} M = A_1 \bigcup_{\varphi} A_2  \quad, \quad A_j = G\times_{K_j} D^2,
\end{eqnarray}
where $\varphi:  G \times_{K_1} \Sph^1  \ra  G \times_{K_2} \Sph^1$
is a $G$-equivariant diffeomorphism. Since $H$ is only acting on the first factor of the product $G \times D^2$, we can write $G \times_{K_j} D^2 = (G/H) \times_{K_j/H} D^2$. Recall that in all our cases $K_j/H$ is a cyclic group, say $\Z_{q_j}$. Since $\varphi$ is $G$-equivariant, it is a bundle map between the fiber bundles
\begin{equation}\label{eqn G/H-bundle} G/H \ra (G/H) \times_{K_j/H} \Sph^1 \ra \Sph^1/\Z_{q_j},
\end{equation}
 for $j=1,2$. Also here, $\varphi$ is completely determined by the image of the path $t \mapsto [(H,e^{2\pi it/q_1})]$ in  $\partial A_1$ for $0 \leq t \leq 1$. Notice that if we take $t \in [0,1]$  fixed, the map $\varphi$ becomes a $G$-equivariant diffeomorphism of $G/H$ on itself, so it is identified with an element  $\gk(t)$ of $N(H)/H$, where we can assume that $\gk(0)=H$. Therefore,
  \begin{equation}\label{eqn clutching 1} \varphi[(H, e^{2\pi it/q_1})]=\lb (\gk(t), e^{2\pi it/q_2})\rb. \end{equation}
  
Before applying this construction to the cyclic and noncyclic isotropy types we compute the fundamental groups of $A_j$ and $A_1 \cap A_2 \simeq \partial(A_1)$. Each component $A_j$ deformation retracts to~$G/K_j$, so the fundamental group of $A_j$ is isomorphic to $K_j$. 
Observe that \mbox{$\Z_{q_j} \ra G/H \times \Sph^1 \ra G/H \times_{\Z_{q_j}} \Sph^1$} is a principal bundle. So, considering the sequence of homotopies of this bundle and of the bundle in \eqref{eqn G/H-bundle} we see that $H$ and $q_j\Z$ are normal subgroups of the fundamental group of $A_1 \cap A_2$ and that $\pi_1(A_1 \cap A_2)\simeq H\rtimes\Z$. 

We claim that the elements of $H$ and $\Z$ in $\pi_1(A_1 \cap A_2)$ commute if $\gcd(q_1,q_2)=1$. 
In fact, in this case the whole subgroup $\Z$ is normal in the fundamental group, therefore, $\pi_1(A_1 \cap A_2) \simeq H \times \Z$, if $H=\Z_d$.

We show next that the isotropy groups are cyclic.

\begin{lem} \label{lem non-cyclic}The 5-dimensional compact simply-connected $G$-manifolds with exceptional orbits only have cyclic isotropy groups.
\end{lem} 
\begin{proof}
By \pref{lem exceptional}, the pair $(H,K)$ is either $(\D_2, \T)$, $(\Dic_2, \T^*)$ or both are cyclic.
We have
\mbox{$M=G/H \times_{\Z_3}D^2 \cup_{\varphi}G/H \times_{\Z_3}D^2$}
 where $\Z_3$ is the quotient $K/H$. We assume $G=\SU(2)$ since $\SO(3)/\D_2 = \SU(2)/\Dic_2$. 
Note that there are exactly two non-equivalent $\Z_3$-actions on $G/H \times D^2$, namely $\xi\cdot(pH,z) = (p\ovl{\xi} H,\xi^{c_j}z)$ for $c_j=1$ or 2 and $\xi = e^{2\pi i/3}$. Recall that $A_j = G/H \times_{(K_j/N)} D^2$. Moreover, the clutching function $\varphi$ is trivial since the normalizer $N(H)/H$ is discrete (see Table \ref{tab SOUOT}).

It is known that $\T^* \simeq \Dic_2 \rtimes \Z_3$ where the $\Z_3$ is generated by $w= -1/2(1+i+j+k)$ and the $\Z_3$-action on $\Dic_2$ is the automorphism that cyclically rotates $i$, $j$ and $k$ (see \cite{coxeter} p. 76). 
The isomorphism between $\Dic_2 \rtimes \Z_3$ and $\T^*$ takes $(x,w)$ to $xw \in \T^*$. So, the action of $\T^*$ on $\SU(2) \times D^2$, which has quotient $G/H \times_{\Z_3} D^2$ is defined by $(x,w)\cdot (p,z) = (pw^{-1}x^{-1},w^{c_j} z)$. 

 Also, the action of $\Dic_2 \rtimes \Z$ on $\SU(2) \times \R$, which has quotient $G/H \times_{\Z_3} \Sph^1$ is given by $(x,a)\cdot (g,s) = (gw^{-a}x^{-1}, s + 2\pi {c_j} a/3)$. Since $\varphi$ is trivial, the induced maps $i_j*: \Dic_2 \rtimes \Z \ra \T^*$ take $(x,a)$ to $(x,w^{ac_j})$ and it is clear that $\pi_1(M) \simeq (\T^**\T^*)/\T^*$ is nontrivial when the quotient is provided by the amalgamation property $i_1(x,a) = i_2(x,a)$ for all $(x,a) \in \Dic_2 \rtimes \Z$ and any choice of $c_1$ and~$c_2$.
\end{proof}

By Lemmas~\ref{lem number except orbits}  and~\ref{lem non-cyclic}, we only need to consider $\SU(2)$-actions with isotropy types $(\Z_{n_1})$, $(\Z_{n_2})$ and  $(\Z_d)$ where $d \mid  \gcd(n_1,n_2)$. To avoid ambiguity assume $n_1 \leq n_2$.  The \mbox{$\SO(3)$-ma}nifolds will be just the $\SU(2)$-actions with ineffective kernel $\Z_2$.


Consider $\Z_{n_1}$ and $\Z_{n_2}$ as subgroups of the same circle parametrized by $t \mapsto e^{2\pi it} \in \SU(2)$, using quaternion notation.  
Let $n_j=dq_j$ and $\xi_j = e^{2\pi i/d q_j}$ be a generator of $\Z_{n_j}$. For the isotropy action in a neighborhood of an exceptional orbit it is easy to see the following.
\begin{lem} The $\Z_{n_j}$-actions on $\SU(2) \times D^2$ are given by
$$\xi_j \cdot (p,z) = (p \ovl{\xi_j}, \xi_j^{da_j} z), $$
for some $a_j$ with $\gcd(a_j, q_j) = 1$ and $0 \leq a_j <q_j$  for $j=1,2$.
\end{lem}

\begin{rem} \label{rmk slice} The isotropy representation at a point with isotropy $\Z_{n_j}$ is determined by the number $a_j$. It is well known that two different representations, $\rho$ and $\rho': \Z_m \ra \O(2)$, are equivalent if and only if $\rho' = \ovl{\rho} $. Notice that they rotate the 2-plane in opposite directions, so $A_j$ is unchanged if we  consider $q_j - a_j$ instead of $a_j$ in the isotropy representation. However, the orientation of the slice in $A_j$ is reversed under this change. 
\end{rem}

A simply-connected $\SU(2)$-manifold $M$ with exceptional orbits, or with only one cyclic isotropy type, is determined by the isotropies,  the  parameters $a_1$ and $a_2$,  and the clutching function $\varphi$. The elements $a_j \in \Z_{q_j}$ have inverse, say $b_j = a_j^{-1}$,  so we write \mbox{$M = M(b_1,b_2,\varphi)$}.

For $d \geq 3$ if we consider an orientation on the manifold and on $G$, it naturally defines an orientation on $G/\Z_d$ and on the slices through the exceptional orbits. Therefore, by \rref{rmk slice}, the $G$-manifolds $M(b_1,b_2,\varphi)$ and $M(q_1 - b_1,b_2, \varphi)$ cannot be equivariantly diffeomorphic, although they have equivalent slice  representations.

The proof of the next result is inspired by Theorem 5.1 of \cite{Bredon}.
\begin{prop}\label{prop k} The $\SU(2)$-manifolds $M(b_1,b_2,\varphi_o)$ and $M(b_1,b_2,\varphi_1)$ are equivalent if and only if the clutching functions $\varphi_o$ and $\varphi_1$ are homotopic.
\end{prop}
\begin{proof}
Call $M_o = M(b_1,b_2, \varphi_o)$ and $M_1 = M(b_1,b_2, \varphi_1)$. Let $H$ be a homotopy between $\varphi_o$ and~$\varphi_1$. Define $F:\partial A_1 \times I \ra \partial A_2 \times I$ by $F(p,t) = (H(p,t),t)$ and the $G$-manifold 
$$N = A_1 \times I \, \bigcup_F \, A_2 \times I,$$
with trivial action on the intervals. This makes $N$ a cylinder with $M_o$ on the bottom and~$M_1$ on the top. Observe that $N/G$ is homeomorphic to $\Sph^2 \times I$ and that if $\pi$ is the projection of $N$ in $\Sph^2 \times I$, then $\pi^{-1}(\Sph^2 \times \{0\}) = M_o$, thus the Covering Homotopy Theorem (see, e.g.~\cite{Bredon}~p.~93) asserts that $N$ is equivalent to $M_o \times I$ (the product of $G$-manifolds with trivial action on the interval) and, therefore, $M_o$ and $M_1$ are equivalent.

Conversely, let $f:M_o \ra M_1$ be a $G$-equivariant diffeomorphism.
 Assume that the manifolds are both written as above using the Slice Theorem and that $f$ restricted to $A_1$ is the identity map. 
Considering $\xi_j(t) = e^{2\pi it/n_j}$,  the clutching functions are
$$\varphi_{i}[(\Z_d, \xi_1(t)^d)] = \lb (\gk_i(t), \xi_2(t)^d) \rb,$$
for $i=0$ or 1, as in (\ref{eqn clutching 1}). So, $f\mid_{\SU(2)\times_{\Z_{n_2}} \Sph^1}$ is an $\SU(2)$-equivariant diffeomorphism of $\SU(2)\times_{\Z_{n_2}} \Sph^1$ given by 
\begin{equation}\label{eqn homotopy}\lb ( \Z_d, \xi_2(t)^d) \rb \mapsto \lb (\gk_o(t)^{-1}\gk_1(t), \xi_2(t)^d) \rb,\end{equation}

 that extends equivariantly to $\SU(2) \times_{\Z_{n_2}} D^2$.  Since the slice representation is a~parametrization of the orbit space this  extension must take the slice $\lb (\Z_d, s\xi_2(t)^d) \rb$ to $\lb ( h_s(t), s\xi_2(t)^d) \rb$, for some $h_s(t) \in \SU(2)/\Z_d$ where $s \in [0,1]$,  $h_1(t) = \gk_o(t)^{-1}\gk_1(t)$ and observe that $h_o(t)$ does not depend on~$t$ since \mbox{$\lb (h_o(t), 0) \rb = f(\lb (\Z_d, 0) \rb)$}. So, the path $\gk_o^{-1}\gk_1$ on $N(\Z_d)/\Z_d$ is homotopically trivial, and therefore the clutching functions $\varphi_o$ and $\varphi_1$ are homotopic.
\end{proof}

For $d=1$  the clutching function $\varphi$ has only one homotopy class, since $N(H) = \SU(2)$, so $M$ can be represented by $M(b_1,b_2)$. We have seen in Remark \ref{rmk slice} that the $\SU(2)$-manifolds $M(b_1,b_2)$ and $M(b_1',b_2')$ are equivalent if and only if $b_j' = b_j$ or $b_j' = n_j - b_j$ for both $j=1$ and~2 simultaneously. 
The manifold  $M(b_1,b_2)$ is equivalent to $\fN_{n_1,n_2}^l$ when $l = b_1n_2 + b_2n_1$ since the isotropy representations coincide as one can see from \pref{prop invariant}. Thus \tref{thm exceptional SU(2)} is proved if $\gcd(n_1,n_2)=1$.

For $d\geq 2$, we analyze the clutching function $\varphi$ in more detail. The path $\gk$ defined in (\ref{eqn clutching 1}) must satisfy $\gk(1)=\xi_1^{b_1}\xi_2^{b_2}\gk(0)$, since $\varphi$ is $\SU(2)$-equivariant. This follows from 
 \begin{equation}\label{eqn n(t)} \lb (\xi_1^{b_1}\gk(0),1)\rb = \varphi[(\xi_1^{b_1}\Z_d, 1)] =  \varphi[(\Z_d,\xi_1^d)] = \lb(\gk(1),\ovl{\xi_2}^d)\rb = \lb(\ovl{\xi_2}^{b_2}\gk(1),1)\rb.
 \end{equation}
Recall our notation $\xi_j(t) = e^{2\pi i t/n_j}$. By  \pref{prop k} we can assume that the path $\gk$ is given by \mbox{$\gk(t)=\xi_1(t)^{b_1}\xi_2(t)^{b_2 + kq_2}\Z_d$} with $k \in \Z$. Therefore
$$\varphi[(\ovl{\xi_1(t)}^{b_1}\Z_d\, ,\, \xi_1(t)^{d})] = \lb(\xi_2(t)^{b_2}e^{2\pi i k/d}\Z_d\, , \, \ovl{\xi_2(t)}^{d})\rb,$$
Thus the homotopy class of $\varphi$ is precisely represented by $k$.

 Notice that for $\mu(t) = e^{2\pi it/dq_1q_2} \in \SU(2)$ and $l = b_1q_2 + b_2q_1 + kq_1q_2$ the clutching function has the form
\begin{equation}\label{eqn clutching}\varphi[(\Z_d\, ,\, \mu(t)^{n_2})] = \lb(\mu(t)^{l}\Z_d\, , \, \ovl{\mu(t)}^{n_1})\rb,
\end{equation}
which is the same expression as in \pref{prop invariant} by changing $l$ by $-l$. The sign does not matter for our purposes since $\fN_{n_1,n_2}^l = \fN_{n_1,n_2}^{-l}$ as observed in \rref{rem sign of l}.

The map $\varphi$ has two homotopy classes if $d=2$ and depends on the number $k\in \Z$ if $d\geq 3$. Thus~$M$ can be represented  by $M(b_1,b_2, \epsilon)$ or $M(b_1,b_2, k)$ respectively, for $\epsilon \in \{0,1\}$ and~$k \in \Z$.

\begin{prop} \label{thm fund gp of M} The fundamental group of $M(b_1,b_2,k)$ is a cyclic group of order $\gcd(n_1,n_2,l)$.
\end{prop}
\begin{proof}
To compute the fundamental group of $M$, we describe the action of \mbox{$\pi_1(\SU(2) \times_{\Z_{n_1}} \Sph^1)$} on the universal covering $\SU(2) \times \R$, such that the quotient is $A_1 \cap A_2=\SU(2) \times_{\Z_{n_1}} \Sph^1$. Take curves $\ga$ and $\gb$ in $A_1 \cap A_2$ that are generators of the fundamental group.
 For each $j=1,2$, we include $A_1 \cap A_2$ in the component $A_j$ by the inclusion $i_j$ and use the \mbox{$\pi_1(A_j)$-action} on the universal covering $\SU(2) \times D^2$ of $A_j$ to regard the loops $\ga$ and $\gb$ as elements of $\pi_1(A_j)$.
  Van Kampen's Theorem asserts that the fundamental group of $M$ is the free product $\pi_1(A_1)*\pi_1(A_2)$ with relations ${i_1}_*[\ga] = {i_2}_*[\ga]$ and ${i_1}_*[\gb] = {i_2}_*[\gb]$, where the maps \mbox{${i_j}_*:\pi_1(A_1\cap A_2) \ra \pi_1(A_j)$}, for $j=1,2$ are induced by $i_j$ on the fundamental groups.

The action of the fundamental group $\pi_1(A_1\cap A_2) \simeq \Z_d \times \Z$ on $\SU(2) \times \R$ with quotient \mbox{$\SU(2) \times_{\Z_{n_1}} \Sph^1$} is given by 
$$(\ovl{u},k)\cdot (p,s) = (p \xi_1^{-(uq_1 + kb_1)}, s + 2\pi k/q_1),$$
where $\Z_d = \{\ovl{0}, \ovl{1}, \cdots, \ovl{d-1}\}$ and $a_1b_1 + u_1q_1 = 1$.
Observe that the  quotient \mbox{$(\SU(2) \times \R) /(\Z_d \times \Z)$} is exactly the orbit space  $(\SU(2) \times \Sph^1)/\Z_{n_j}$. Indeed, since  $\gcd(b_1,q_1)=1$, for any $0\leq r < n_1$ there are integers $u$ and $k$ with $0 \leq u < d$ such that $r = uq_1 + kb_1$. So, the $\Z_{n_1}$-action can be written~as 
$$ \xi_1^r\cdot(p,e^{is}) = (p\xi_1^{-(uq_1 + kb_1)}, \exp(s + 2\pi ua_1 +2\pi ka_1b_1/q_1)) =
 (p\xi_1^{-(uq_1 + kb_1)}, \exp(s +  2\pi k/q_1)),  $$
that clearly defines the same quotient as $(\SU(2) \times \R) /( \Z_d \times \Z)$.

It is convenient to define the loops in $\SU(2)/\Z_d \times_{\Z_{q_1}}\Sph^1$ that generate its fundamental group, $\Z_d \times \Z$. The loop $\ga: [0,1] \ra A_1 \cap A_2$ defined by $\ga(t) = [(\xi_1(t)^{-q_1}\Z_d, 1)]$ corresponds to $(\ovl{1},0)$ in $\Z_d \times \Z$. In fact, let $\wt{\ga}(t) = (\xi_1(t)^{-q_1},0)$ be a lifting of $\ga$ by $(1,0) \in \SU(2) \times \R$. So, $\wt{\ga}(1) = (\xi_1^{-q_1},0) = (\ovl{1},0)_{\Z_d \times \Z}\cdot(1,0) = (\ovl{1},0)_{\Z_d \times \Z}\cdot \wt{\ga}(0)$. 

On the other hand, the loop $\gb:[0,1] \ra A_1\cap A_2$ defined by $\gb(t) = [(\xi_1(t)^{-b_1}\Z_d, \xi_1(t)^d)]$ corresponds to $(\ovl{0},1) \in \Z_d \times \Z$. In fact, consider $\wt{\gb}(t) = (\xi_1(t)^{-b_1}, 2\pi t/q_1)$, a lifting of $\gb$ to $\SU(2) \times \R$ by $\wt{\gb}(0) = (1,0)$. Then $\wt{\gb}(1)=  (\xi_1^{-b_1}, 2\pi/q_1 ) =(\ovl{0},1)_{\Z_d \times \Z} \cdot \wt{\gb}(0)$. So, $\gb$ corresponds to $(\ovl{0},1) \in \Z_d \times \Z$.

If $e_j$ is a generator of $\Z_{n_j}$ in the free product $\Z_{n_1} * \Z_{n_2}$, the induced loop $(i_1\circ \ga)(t)$ corresponds to $e_1^{q_1}$. The loop $(i_1 \circ \gb)$ in $A_1$ corresponds to $e_1^{b_1} \in \Z_{n_1}$. In fact, the lifting of $i_1 \circ \gb$ by the point $(1,1)$ of $\SU(2) \times D^2$ is the curve $t \mapsto (\xi_1(t)^{-b_1}, \xi_1(t)^d)$. Then, 
$$\xi_1^{b_1} \cdot \wt{(i_1 \circ \gb)}(0) = \xi_1^{b_1} \cdot (1,1) = (\ovl{\xi_1}^{b_1}, \xi_1^d) = \wt{(i_1 \circ \gb)}(1).$$
That is, ${i_1}_*:\Z_{d} \times \Z \ra \Z_{n_1}$ takes $(\ga^u,\gb^n)$ to $e_1^{uq_1 + nb_1}$.

We need to include the loops $\ga$ and $\gb$ in $A_2$. To do this we simply use the composition,~$i_2$, of the clutching function $\varphi$ with the inclusion of $\SU(2)/\Z_d \times_{\Z_{q_2}} \Sph^1$ in $A_2$. 
 The induced loop $(i_2 \circ \ga)(t) = \lb(\xi_2(t)^{-q_2}\Z_d, 1)\rb$ has lifting by $(1,1) \in \SU(2) \times D^2$, with end point $(\xi^{-q_2}, 1) = \xi_2^{q_2} \cdot (1,1)$. So, $(i_2 \circ \ga)$ corresponds to $e_2^{q_2} \in \Z_{n_2}$. 
The loop $(i_2 \circ \gb)(t) = \lb(\xi_2(t)^{(b_2 + kq_2)}\Z_d, \xi_2(t)^{-d})\rb$  has lifting   $t \mapsto(\xi_2(t)^{(b_2 + kq_2)}, \xi_2(t)^{-d})$ that starts at \mbox{$(1,1) \in \SU(2) \times D^2$} and that ends at $(\xi_2^{(b_2 + kq_2)}, \xi_2^{-d}) = \xi_2^{(-b_2 -  kq_2)}\cdot(1,1)$, so the loop corresponds to $e_2^{-(b_2+q_2k)} \in \Z_{n_2}$. Therefore the map ${i_2}_*$ takes $(\ga^u, \gb^n)$ to $e_2^{(u-kn)q_2-nb_2}$.
We conclude that the fundamental group of $M$ is generated by $e_1$ and $e_2$ with the relations $e_1^{n_1} =e_2^{n_2} = 1$, $e_1^{q_1} = e_2^{q_2}$ and $e_1^{b_1}=e_2^{-b_2 - kq_2}$. From these three identities and  $\gcd(b_j, q_j) = 1$ we see that \mbox{$\pi_1(M) \simeq \Z_{\gcd(n_1,n_2, b_1q_2 + b_2q_1 +kq_1q_2)}$}.
\end{proof}

\begin{rem} As a consequence of \pref{thm fund gp of M}, $d=\gcd(n_1,n_2)$ when the $\SU(2)$-manifold is simply-connected. In fact,  $\gcd(n_1,n_2,l)=1$ and $\gcd(b_j,q_j)=1$ imply that \mbox{$\gcd(q_j, l)=1$}, therefore $\gcd(q_1,q_2)=1$. So, the $\SU(2)$-action depends on a triple of integer parameters that belongs to $\fp=\{(b_1,b_2,k) \in \Z^3\,:\, 0\leq b_j<q_j \, , \, (b_j,q_j)=1\, ,\, j=1, 2\}$.\end{rem}

The proof of the following result is a simple computation, then it will be omitted.
\begin{prop}\label{claim bijection} The map $l(b_1,b_2,k)=b_1q_2 + b_2q_1 + kq_1q_2$ is a bijection between the sets \mbox{$\fp = \{(b_1, b_2, k) \in \Z^3 : 0 \leq b_j < q_j  ,  (b_j,q_j)=1, j=1, 2 \}$} and \mbox{$\fq=\{ l \in \Z : (l,q_j)=1,\, j=1, 2\}$}. 
\end{prop}

We will show now that the $G$-manifold $M(b_1,b_2,k)$ is equivariantly diffeomorphic to $\fN_{n_1,n_2}^l$ where \mbox{$l = b_1q_2 + b_2q_1 + kq_1q_2$}. For the next result recall that we have defined $\gcd(0,0)=1$.

 \begin{thm}\label{lem equivariance} Let $n_1 \leq n_2$ be positive integers with $d = \gcd(n_1,n_2)\geq 2$, set $q_j = n_j/d$ and take  $b_j \in \Z$ coprime with $q_j$ satisfying $0 \leq b_j < q_j$, for $j=1,2$. Let $k$ be an integer for $d \geq 3$ or $k \in \Z_2$ for $d=2$. Then, the $\SU(2)$-manifold $M(b_1,b_2,k)$ is equivariantly diffeomorphic to $\fN_{n_1,n_2}^l$, where $l = b_1q_2 + b_2q_1 + kq_1q_2$. Moreover, these $\SU(2)$-manifolds are pairwise nonequivalent except for $M(b_1,b_2,k) = M(q_1 - b_1, q_2 - b_2, -k - 2)$.
 \end{thm}
 \begin{proof}
It is clear that $M(b_1,b_2,k)$ is determined by the isotropy representations around the exceptional orbits and the homotopy class of the clutching function $\varphi$. So, \mbox{$M(b_1,b_2,k) = \fN_{n_1,n_2}^l$} with $l = b_1q_2 + b_2q_1 +kq_1q_2$ since by \pref{prop invariant} they coincide in both representations and also have the same clutching function, up to homotopy. Moreover, Propositions \ref{thm fund gp of M} and \ref{claim bijection} show that each $l$ is reached exactly once by that formula.

For $d=2$ we know that $k\in\{0,1\}$ since the homotopy class of the clutching function is defined modulo 2. We use the identity $M(b_1,b_2,k) = \fN_{n_1,n_2}^l$ and  $\fN_{n_1,n_2}^{-l} = \fN_{n_1,n_2}^l$  (see, \rref{rem sign of l}) to conclude that $M(b_1,b_2,k) = M(q_1 - b_1,q_2 - b_2,-k-2) = M(q_1 - b_1,q_2 - b_2,k)$ for $k=0$ or~1. \rref{rmk slice} shows that  otherwise these $\SU(2)$-manifolds are pairwise distinct.

For $d\geq 3$ two $G$-manifolds $M(b_1,b_2,k)$ are equivalent if and only if they have the same number \mbox{$l = b_1q_2 + b_2q_1 + kq_1q_2$}, see \pref{claim fixed}. 
This and \rref{rmk slice} imply that $M(b_1,b_2,k)$ is equivalent to $M(b_1',b_2',k')$ if and only if the parameters are exactly the same, or $b_j'=q_j - b_j$ and $k'=-k-2$. This corresponds to replace $l$ by $-l$ in~$\fN_{n_1,n_2}^l$.
\end{proof}

Observe that it is a consequence of the discussion in the proof of \tref{lem equivariance} that the manifolds $\fN_{n_1,n_2}^{l}$ and  $\fN_{n_1,n_2}^{l'}$ are equivalent if, and only if, $l \equiv l' \mod 2q_1q_2$ for $d = 2$. This concludes the proof of Theorem \ref{thm exceptional SU(2)}.

\section{Actions with singular orbits and the proof of Theorem D}\label{sec singular}

For simply-connected compact $G$-manifolds of dimension 5 with singular orbits the number of orbit types cannot exceed 3, c.f., \lref{lem restrictions} below.
In \sref{sec 2OT} we use a classical result by Bredon~\cite{Bredon}, Hsiang and Hsiang~\cite{Hsiang1967} and Janich~\cite{Janich1966} to classify the actions with exactly two orbit types. In \sref{sec 3OT} we make a few comments about Hudson's work~\cite{Hudson1977} on $\SO(3)$-manifolds with three orbit types.

The classification of $\SO(3)$-actions on simply-connected 5-manifolds with singular orbits and cohomogeneity 2 was carried out by Hudson in \cite{Hudson1977}. The actions with cohomogeneity~3 were discussed in the same paper, but the $\SO(3)$-manifolds in \eref{ex kmissed} were overlooked.
For the sake of completeness we classify again the \mbox{$\SO(3)$-manifolds} without fixed points since it can be obtained together with the $\SU(2)$ case. 

The following lemma gives strong restrictions to the possible chains of isotropy subgroups. It is inspired by Lemma 1A in \cite{Hudson1977}.

\begin{lem} \label{lem restrictions} Let $M$ be a 5-dimensional simply-connected $G$-manifold with singular orbits. 
\begin{enumerate}[(a)]
\item \label{itemlem 2}If the action has exactly two orbit types, say $(H) \leq (K)$, then the pair of principal and singular isotropy groups $(H,K)$ is $(\Z_m,\SO(2))$, $(\D_m, \O(2))$ or $(\SO(2), \SO(3))$ if \mbox{$G=\SO(3)$} and $(\Z_m,\SO(2))$ or $(\{1\},\SU(2))$, if $G=\SU(2)$, for $m \geq 1$;

\item \label{itemlem 3} If the action has three orbit types then the isotropy types are $\Z_2 \times \Z_2 \subset  \O(2) \subset  \SO(3)$. There is no $\SU(2)$-action with three isotropy types;
\item \label{itemlem 4} Neither $\SO(3)$ nor $\SU(2)$ acts on $M$ with more than three orbit types.
\end{enumerate}
\end{lem}
\begin{proof}[Sketch of the proof] In general, for a singular point $p \in M^n$, if $K = G_p$ and $k$ is the codimension of $G(p)$, it is known that the slice action $K \hra \R^k$, of the isotropy group $K$ on the tangent space of a slice at $p$, is a linear action which has the same isotropy structure as the action $G \hra M$ in a neighborhood of $p$. So the chain of isotropy types $(H) \leq \cdots \leq (K)$ is possible for a $G$-action on $M$ only if there is a representation $\rho:K \ra \O(k)$ such that the action $\rho(K) \hra \R^k$ has the same chain of isotropy types. Now, the proof of \ref{itemlem 2} and \ref{itemlem 3} follow of a case-by-case analysis in the subgroups of $\SO(3)$ and $\SU(2)$ and its representations in $\O(3)$.

The subgroups of $G$ have dimensions either zero, one or three. If there are three isotropy types and $(G)$ is not one of them, two of them have the same dimension. 
Since there are no exceptional orbits, there are two one-dimensional isotropy types, say $(K)$ and $(K')$. Let $M_{(K)}$ and $M_{(K')}$ be the set of points in $M$ whose isotropy groups belong to $(K)$ and~$(K')$, respectively. Both sets $M_{(K)}$ and $M_{(K')}$ are submanifolds of $M$ and are projected to the boundary  of the disk in the quotient. 
So there is a point $p \in M$ such that in any neighborhood of $p$ there are orbits of type $(G/K)$ and $(G/K')$, but there is no representation neither of $\O(2)$ nor $\Pin(2)$ on $\O(3)$ with isotropy $\SO(2)$. 
This shows that if the action has more than two distinct orbit types, then there is a fixed point. It also shows that there is no 5-dimensional $G$-action with more than three orbit types and the lemma follows.
\end{proof}

\subsection{Actions with singular orbits and two orbit types}\label{sec 2OT}


The Second Classification Theorem in \cite{Bredon}, p. 257 and p. 331 shows, in particular, that given a contractible topological manifold $X$  with non-empty boundary and  two subgroups $H \subset K$ of a group $G$ with $\dim H < \dim K$, the set of classes of $G$-manifolds with orbit space $X$ and isotropy types $(H)$ and $(K)$ is in one-to-one correspondence with the quotient
\begin{equation}\label{eqn action homotopy} \pi_{n-1}\left(N(H)/(N(H) \cap N(K)) \right)\,/\,\pi_o\left(N(H)/H \right).\end{equation}
For our purposes an explicit expression for the $\pi_o(N(H)/H)$-action above is not needed  since the groups involved in (\ref{eqn action homotopy}) are quite simple, see Table \ref{tb two isot}.

\begin{table}[hbt!]
\begin{center}
\begin{tabular}{ccccccc}
&H & K & G & $\pi_{n-1}\left(N(H)/(N(H) \cap N(K)\right)$ & $\pi_o(N(H)/H)$ & $l$  \\
\hline

{\bf (a)}&$\{1\}$ & $\SO(2)$ & $\SO(3)$ & $\pi_1(\RP^2)$ & 1      & 2\\

{\bf (b)}&$\{1\}$ & $\SO(2)$ & $\SU(2)$ & $\pi_1(\RP^2)$ & 1      & 2 \\
\hline

{\bf (c)}&$\{1\}$  & $\SU(2)$   & $\SU(2)$ & $\{0\}$        & 1 & 1  \\
\hline

{\bf (d)}&$\Z_2$ & $\SO(2)$ &  $\SO(3)$ &   $\{0\}$    & 1 & 1\\
%
{\bf (e)}&$\Z_2$ & $\SO(2)$ &  $\SU(2)$ &   $\pi_1(\RP^2)$ & $1$ & 2 \\
\hline

{\bf (f)}&$\Z_m$ & $\SO(2)$ &  $\SO(3)$ &   $\{0\}$    & 1 & 1 \\

{\bf (g)}& $\Z_m$ & $\SO(2)$ &  $\SU(2)$ &  $\{0\}$     & 1 & 1 \\
\hline

{\bf (h)} & $\Z_2$ & $\O(2)$ & $\SO(3)$   &$\pi_1(\Sph^1)$ & $\Z_2$ & $\Z$ \\
\hline

{\bf (i)} & $\D_m$  & $\O(2)$   & $\SO(3)$ & $\{0\}$        & $\Z_2$& 1  \\
\hline

{\bf (j)} & $\SO(2)$  & $\SO(3)$   & $\SO(3)$ & $\{0\}$        & 1 & 1  \\
\hline
\end{tabular}
\caption{\label{tb two isot} {\small $l$ is an upper bound for the number of actions with 2 isotropy types $H\subset K$}}
\end{center}
\end{table}
\vspace{-0,6cm}

Due to Lemma \ref{lem restrictions} and the Second Classification Theorem  we obtain an upper bound $l$ (see Table~\ref{tb two isot}) for the number of \mbox{$G$-manifolds} with exactly two orbit types and quotient a 2 or 3-disk, hence the classification of $G$-actions with contractible orbit space is complete by showing that we have as many examples as possible. The examples below represent the corresponding  enumeration in Table~\ref{tb two isot}.

\clearpage
\noindent {\it Examples.}
\begin{enumerate}[(a)]
\item {\it The two $\SO(3)$-actions with isotropies $\{1\}$ and $\SO(2)$.} They  are linear actions on $\Sph^5$  and $\Sph^3 \times \Sph^2$ described in Examples \ref{ex sphere} (\ref{ex linear 2OT}) and \ref{ex 1missed} (\ref{ex 1missed item Id S1}), the last one is $\fN_{0,2}^1$.
%


\item {\it The two $\SU(2)$-actions with isotropies $\{1\}$ and $\SO(2)$.} They are $\fN_{0,1}^1 = \Sph^3 \wt{\times} \Sph^2$  and the action on \mbox{$\fW = \SU(3)/\SO(3)$} given by  $B\cdot [C] = [\diag(B,1)C]$, described in \eref{ex SU(2) on W}. 

\item {\it The linear $\SU(2)$-action on $\Sph^5$ with isotropies $\{1\}$ and $\SU(2)$.} See \eref{ex sphere} (\ref{ex SU2 S5}).

\item[(d)]\hspace{-0,2cm}, (e), (f) and (g) {\it The $G$-action with isotropies $\Z_m $ and $\SO(2)$ with  $m \geq 2$.} They are $\fN_{0, m}^1$. Recall that $\SU(2)$-actions with $H=\Z_{2m}$ are effective $\SO(3)$-actions with principal isotropy $\Z_m$ (see \sref{sec main}).

\item[(h)] and (i) {\it The $\SO(3)$-actions with isotropy types $\Z_2$ and $\O(2)$ or $\D_m$ and $\O(2)$.} Remark~3C in~\cite{Hudson1977} asserts that these examples are not simply-connected for any~$m$.

\item[(j)] {\it The linear $\SO(3)$-action on $\Sph^5$ with isotropy types $\SO(2)$ and $\SO(3)$.} See \eref{ex sphere} (\ref{ex SO(3) on S5}).
\end{enumerate}

If the orbit space is three dimensional then the quotient is a topological 3-manifold with boundary, c.f. \pref{prop quotient} (\ref{B3}). In this case, the isotropy groups are $\SO(2)$ and $\SO(3)$ (see \lref{lem restrictions}). In a neighborhood of a fixed point, the $\SO(3)$ slice action on $\R^5$ is the one with two fixed directions, thus $M/G$ is smooth and has non-empty boundary.  When \mbox{$K = G$}, the Second Classification Theorem guarantees that even for a non-contractible topological manifold $X$,  there is only one class of $G$-manifolds with isotropy types $(H)$ and~$(G)$, and orbit space $X$. Finally, the next proposition completes the possible orbit spaces for a cohomogeneity 3 action.

\begin{prop}\label{prop 3M bound}
 A simply-connected compact 3-dimensional manifold with boundary is diffeomorphic to a 3-sphere with $k$ open 3-disks removed.
\end{prop}
\begin{proof}
Let $X$ be a simply-connected compact 3-manifold with boundary. We claim that the boundary of $X$ is a disjoint union of 2-spheres. In fact, Poincar\'e duality to the pair $(X, \partial X)$ guarantees that $H_2(X,\partial X) \simeq H^1(X)$, so from the exactness of the relative homology sequence, 
$$\cdots \ra H_2(X, \partial X) \ra H_1(\partial X) \ra H_1(X) \ra \cdots$$
we conclude that $H_1( \partial X) = 0$. So each connected component of $\partial X$ is homeomorphic to $\Sph^2$ by the classification of compact surfaces.

The manifold obtained from $X$ by covering each connected component with a 3-disk is simply-connected compact without boundary, so it is a 3-sphere and the proposition is proved.
\end{proof}
Summarizing the discussion above we have the following.
\begin{prop}
There is precisely one $\SO(3)$-action (up to equivalence) with isotropy types $(\SO(2))$ and $\SO(3)$ with quotient a 3-sphere with $k$ three-disks removed, $k>0$.
\end{prop}
These actions are precisely those described in \eref{ex kmissed}.
This concludes the classification of actions with singular orbits and two orbit types. The remaining part of the proof of \tref{thm main intro} follows from~\cite{Hudson1977} since only \mbox{$\SO(3)$-manifolds} have three isotropy types.

\subsection{Actions with three orbit types} \label{sec 3OT}
By \lref{lem restrictions}, the $G$-actions with three orbit types have isotropy groups $\Z_2 \times \Z_2$, $\O(2)$ and $\SO(3)$, thus $G=\SO(3)$. So, all these actions have fixed points and the isotropy representation around a fixed point is the irreducible $\SO(3)$-action on $\R^5$. From this representation it is easy to conclude that the fixed points are isolated and there are at least two fixed points.
 Hudson~\cite{Hudson1977}  showed that if there are exactly two fixed points, it is the linear action on $\Sph^5$ described in \eref{ex sphere} (\ref{ex irreducible}).
 If it has three fixed points then the action is equivalent to left multiplication on the cosets of the Wu-manifold $\fW=\SU(3)/\SO(3)$ (see \eref{ex SO(3) on Wu}). Moreover, there are two $\SO(3)$-manifolds with exactly four fixed points: the Brieskorn variety,~$\fB$ of type $(2,3,3,3)$, and the connected sum of two Wu-manifolds with the  action above.
All other examples of $\SO(3)$-manifolds with three orbit types have more than four fixed points.  In \cite{Hudson1977} it is proved that the $\SO(3)$-manifolds with more than two isolated fixed points are connected sums of copies of $\fB$ and~$\fW$ (see \eref{ex hudson}).

\section{Five-manifolds with nonnegative curvature}\label{sec positive}
In this section we prove Theorems \ref{thm metric} and \ref{thm positive}. \tref{thm positive} is a consequence of Theorems~\ref{thm exceptional SU(2)} and~\ref{thm main intro}, by using Frankel's Lemma \cite{frankel} and the classification of the $G$-manifolds $M$ with fixed point set with codimension one or two in $M/G$ (c.f.,~\cite{grove_kim} and~\cite{grove_searle2}).
In our context, the following lemma provides a more elementary proof.

\begin{lem}\label{isolated fixed M5} Let $M$ be a simply-connected compact $\SO(3)$-manifold of dimension~5. If $M$ admits an invariant metric of nonnegative (resp. positive) curvature, then the number of isolated fixed points cannot exceed~3 (resp. 2).
\end{lem}
\begin{proof} If the action has isolated fixed points, then the quotient is a topological 2-disk since the isotropy action in a neighborhood of an isolated fixed point is equivalent to the irreducible $\SO(3)$-representation in $\SO(5)$, and thus, the isotropy types are $(\Z_2 \times \Z_2)$, $(\O(2))$ and $\SO(3)$. The orbit strata of a group action is well known to be totally geodesic (see e.g. \cite{grove-survey}). Hence the boundary of $M/\SO(3)$ consists of a geodesic polygon with $n$ edges and $n$ vertices. The edges correspond to the singular isotropies~$\O(2)$ and vertices are the fixed points. From the isotropy representation it follows that the angles between the edges are all equal to $\pi/3$  (see \eref{ex sphere} (\ref{ex irreducible})).

By O'Neill's formula the (interior of the) quotient inherits a metric of nonnegative (resp. positive) curvature if $M$ has an invariant metric of nonnegative (resp. positive) curvature. Thus, by the  Gauss-Bonnet Theorem, the sum of inner angles of the $n$-polygon $M/\SO(3)$ is equal to, or bigger (resp. strictly bigger) than $\pi(n-2)$. So, $n=2$ or 3 if the curvature is nonnegative and $n=2$ if the curvature is positive.
\end{proof}

\begin{proof}[Proof of \tref{thm metric}]
By \tref{thm main intro} and \sref{sec 3OT} the $G$-manifolds with more than 3 isolated fixed points are $k\, \fW \, \# \,l \, \fB$ with $(k,l) \neq (1,0)$ (since $\SO(3) \hra \fW$ has exactly 3 fixed points). So, by \lref{isolated fixed M5} they do not admit invariant metrics with nonnegative curvature.

The connected sum $M$ of $k$ copies of $\Sph^3 \times \Sph^2$ (see \eref{ex kmissed}) has quotient $X$ diffeomorphic to $\Sph^3$ with $k+2$ three-disks removed. If $M$ admits a metric of nonnegative curvature, then~$X$ with the induced metric also has nonnegative curvature. As follows from the proof of the Soul Theorem~\cite{cheeger}, a compact nonnegatively curved manifold $X$ with non-empty convex boundary contains a totally geodesic compact submanifold $\gS$ without boundary and $\gS$ is a deformation retract of $X$. In our case $\dim\gS \neq 0$ since $X$ is not a disk. Also, $\dim \gS \neq 1$ since $X$ is simply-connected. Thus $\gS$ is a simply-connected surface and a neighborhood of $\gS$ is diffeomorphic to  $\Sph^2 \times (-1,1)$. Using the flow of the gradient like vector field in the proof of the Soul Theorem it follows that $\partial X$ has two connected components. Therefore, $k=0$ and only $\Sph^3 \times \Sph^2$ with the linear $\SO(3)$-action on the first factor admits an invariant metric of 
nonnegative curvature.

On the other hand, all the other actions in Theorems \ref{thm exceptional SU(2)} and \ref{thm main intro}, i.e., the linear actions on~$\Sph^5$ and $\Sph^3\times \Sph^2$, the $\SO(3)$ or $\SU(2)$ left multiplication on the cosets in $\fW$, and $\fN_{m,n}^l$ clearly admit an invariant metric of nonnegative curvature.
\end{proof}

\begin{proof}[Proof of \tref{thm positive}]
We restrict ourselves to the actions in \tref{thm metric}. By Theorems \ref{thm exceptional SU(2)} and~\ref{thm main intro}, the $G$-manifolds diffeomorphic to $\Sph^3 \times \Sph^2$ are $\fN_{m,n}^l$ and the $\SO(3)$-manifold in \eref{ex 1missed}. This last one has quotient $X$ diffeomorphic to a 3-sphere with two 3-disks removed, thus its soul is homeomorphic to a 2-sphere and by the Soul Theorem $X$ cannot be positively curved. So neither $\Sph^3 \times \Sph^2$ admits an invariant metric of positive curvature. Also, the $\SO(3)$-action on~$\fW$ has three fixed points  (see \eref{ex SO(3) on Wu}), and therefore does not admit an invariant metric of positive curvature by \lref{isolated fixed M5}.

We finally observe that~$\fN_{m,n}^l$ with $\gcd(m,n)\geq 3$ does not admit a metric of positive curvature. In fact, by \pref{claim fixed} the fixed point set of the principal isotropy group~$\Z_{\gcd(m,n)}$ has two connected components of dimension three and by Frankel's Lemma \cite{frankel} in  a positively curved manifold $M^n$ the sum of the dimensions of two totally geodesic submanifolds cannot exceed $n$.
It is also clear that the fixed point set of the \mbox{$\SO(3)$-manifold}~$\fN_{0,0}^1$ is the disjoint union of two 3-spheres, thus also cannot admit metric of positive curvature.
\end{proof}
\begin{rem}
 Using standard arguments of equivariancy it can be shown that for fixed $m$ and $n$ with $\gcd(m,n)=1$ or 2, only 3 of the manifolds $\fN_{m,n}^l$ are candidates to admit positive curvature.
\end{rem}

\bibliography{BIBLI}
\vspace{0.3cm}

\noindent UNIRIO  - Brazil. \newline
{\small Universidade Federal do Estado do Rio de Janeiro}\newline
E-mail adress: {\it simas@impa.br}

\end{document}